\documentclass[10pt,a4paper]{article}
\usepackage[psamsfonts]{amsfonts}
\usepackage{amsmath,amssymb}
\usepackage{graphicx}
\usepackage{url}

\newcommand{\N}{\mathbb{N}}			
\newcommand{\R}{\mathbb{R}}			
\newcommand{\Z}{\mathbb{Z}}			
\newcommand{\Prob}{\mathbb{P}}		
\newcommand{\ahalf}{\tfrac{1}{2}}	
\newcommand{\Cmax}{C_{\text{max}}}	
\newcommand{\Cle}{C_{\text{le}}}	

\numberwithin{equation}{section}	

\usepackage{theorem}				
\theorembodyfont{\upshape}			

\newtheorem{theorem}{Theorem}[section]
\newtheorem{corollary}[theorem]{Corollary}

\newtheorem{definition}[theorem]{Definition}

\newenvironment{proof}[1][x]
	{\par\vskip\topsep\noindent
	 \ifx#1x{\it Proof.}\else{\it #1.}\fi
	 \hskip 5pt plus1pt minus1pt}
	{\hfill$\Box$\par\vskip\topsep}

\title{Pattern theorems, ratio limit theorems and Gumbel maximal clusters for 
random fields}
\author{Remco van der Hofstad\thanks{Dept of Mathematics and Computer Science, 
Eindhoven University of Technology, P.O. Box 513, 5600\,MB Eindhoven, The 
Netherlands. E-mail: {\tt rhofstad@win.tue.nl}}
\and Wouter Kager\thanks{EURANDOM, P.O.~Box 513, 5600\,MB Eindhoven, The 
Netherlands. E-mail: {\tt kager@eurandom.tue.nl}}}

\begin{document}

\maketitle
\begin{abstract}
	We study occurrences of patterns on clusters of size~$n$ in random fields 
	on~$\Z^d$. We prove that for a given pattern, there is a constant $a>0$ 
	such that the probability that this pattern occurs at most $an$ times on a 
	cluster of size~$n$ is exponentially small. Moreover, for random fields 
	obeying a certain Markov property, we show that the ratio between the 
	numbers of occurrences of two distinct patterns on a cluster is 
	concentrated around a constant value.  This leads to an elegant and simple 
	proof of the ratio limit theorem for these random fields, which states 
	that the ratio of the probabilities that the cluster of the origin has 
	sizes $n+1$ and~$n$ converges as $n\to\infty$. Implications for the 
	maximal cluster in a finite box are discussed.
\end{abstract}

\section{Introduction and main results}
\label{sec:introduction}

We consider random fields on the lattice~$\Z^d$ in dimensions $d\geq2$, with a 
finite state space $S=\{0,1,\dots,q-1\}$ per site for some $q\geq2$. Thus, the 
configuration space for the models studied here is $\Omega = S^{\Z^d}$. This 
space is endowed with a probability measure~$\Prob$, which we assume to be 
translation-invariant. The results in this paper hold under different further 
conditions on the measure~$\Prob$, which we will define and discuss first. To 
do so, for any configuration $\omega \in \Omega$ and any $V \subset \Z^d$, we 
will write $\omega_V$ for the configuration restricted to the set~$V$, that 
is, $\omega_V$ is considered to be an element of $S^V$.

\begin{definition}[Finite-energy property]
	We say that the measure~$\Prob$ has the \emph{finite-energy property} if 
	there exists an $h\in(0,1)$ such that for all states $s\in S$,
	\begin{equation}
	\label{finiteE}
	\begin{split}
		h
		&\leq \inf_{\sigma \in \Omega} \Prob\bigl( \omega(x)=s \bigm| 
		\omega_{\Z^d \setminus \{x\}} = \sigma_{\Z^d \setminus \{x\}} \bigr) 
		\\
		&\leq \sup_{\sigma \in \Omega} \Prob\bigl( \omega(x)=s \bigm|
		\omega_{\Z^d \setminus \{x\}} = \sigma_{\Z^d \setminus \{x\}} \bigr)
		\leq 1-h.
	\end{split}
	\end{equation}
\end{definition}

\begin{definition}[Markov property]
	We say that the measure~$\Prob$ has the \emph{Markov property} if the 
	state of a site $x\in\Z^d$ depends only on the states of its nearest 
	neighbours in~$\Z^d$ and not on the rest of the field, that is, if for all 
	$x\in\Z^d$ and $s\in S$,
	\begin{equation}
	\label{Markov}
		\Prob\bigl( \omega(x) = s \bigm| \omega_{\Z^d \setminus \{x\}} \bigr)
		= \Prob\bigl( \omega(x) = s \bigm| \omega_{N_x} \bigr),
	\end{equation}
	where $N_x = \{ y\in\Z^d : |y-x| = 1 \}$ is the set of neighbours of~$x$.
\end{definition}

Observe that if the random field is Markovian, then by translation invariance 
and~\eqref{Markov} it has the finite-energy property~\eqref{finiteE} if and 
only if for a given vertex~$x$, each state~$s$ has strictly positive 
probability regardless of the states of the neighbours of~$x$. The Markov 
property implies the following ``boundary-$s$ Markov property'' for every 
state $s\in S$, which will actually be sufficient for our purposes:

\begin{definition}[Boundary-$s$ Markov property]
	For a finite $V\subset\Z^d$, write $N_V = \bigcup_{x\in V} (N_x\setminus 
	V)$ for the set of neighbours of~$V$. We say that~$\Prob$ has the 
	\emph{boundary-$s$ Markov property} if for every finite $V\subset\Z^d$, 
	given that all sites of~$N_V$ are in the state $s\in S$, the configuration 
	on~$V$ is conditionally independent of the configuration on $\Z^d\setminus 
	(V\cup N_V)$.
\end{definition}

A special case of the boundary-$s$ Markov property is the so-called 
\emph{empty-boundary Markov property} as defined in~\cite{Grim06}, which is 
the boundary-$0$ Markov property in our terminology.

Note that in the discussion so far, we have considered Markovian properties 
under nearest-neighbour dependencies only. Indeed, in the conditional 
probabilities~\eqref{Markov} we only consider nearest neighbours of~$x$, and 
in our formulation of the boundary-$s$ Markov property, it suffices that the 
nearest neighbours of~$V$ are in the state~$s$ to have independence 
between~$V$ and the outside world. It is not difficult, however, to extend the 
methods in this paper to random fields with more general dependencies between 
sites, as long as these dependencies have finite range.

As a typical example of the kind of random fields we want to study, we 
consider site percolation. In this case, $q=2$ and for given $0<p<1$, we take 
the measure~$\Prob$ to be the Bernoulli product measure $\Prob_p$ such that 
$\Prob_p(\omega(x) = 1) = p$ and $\Prob_p(\omega(x) = 0) = 1-p$ for each 
$x\in\Z^d$. Since the states of different sites are independent in this case, 
it is clear that~\eqref{finiteE} holds with $h = \min(p, 1-p)$, and the field 
is obviously Markovian as well. For a given percolation configuration $\omega 
\in \Omega$, we say that a site $x\in\Z^d$ is \emph{occupied} if $\omega(x) = 
1$ and \emph{vacant} if $\omega(x) = 0$. For convenience, the same terminology 
will be used for the states $0$ and~$1$ of any random field. This does not 
attach any special meaning to the states $0$ and~$1$, since one can always 
study the random fields obtained by permuting the states of~$S$.

Using the terminology introduced above, we write $C(x) = C(x,\omega)$ for the 
occupied cluster of the site $x \in \Z^d$. That is, $C(x)$ is the set of 
occupied sites (i.e., sites in state~$1$) that can be connected to~$x$ by a 
nearest-neighbour path passing only through occupied sites. By convention, we 
set $C(x)= \varnothing$ if $\omega(x) \neq1$ and we write $C = C(0)$ for the 
cluster of the origin. In general, if $X \subset \Z^d$, then the number of 
sites in~$X$ is denoted by $|X|$. In particular, the number of sites in the 
occupied cluster of the origin will be denoted by~$|C|$. The results in this 
paper hold under the assumption that the distribution of $|C|$ has an 
exponential or stretched exponential tail.

\begin{definition}[Exponential tail]
	We say that the cluster-size distribution has an \emph{exponential tail} 
	if the limit
	\begin{equation}
	\label{limit}
		\mu = \lim_{n\to\infty} \big[ \Prob(|C|=n) \big]^{1/n} = 
		\lim_{n\to\infty} \big[ \Prob(n \leq |C| < \infty) \big]^{1/n}
	\end{equation}
	exists for some $0<\mu\leq1$.
\end{definition}

\begin{definition}[Stretched exponential tail]
	We say that the cluster-size distribution has a \emph{stretched 
	exponential tail with exponent $\beta\in(0,1)$} if the limit
	\begin{equation}
	\label{superlimit}
		\nu = \lim_{n\to\infty} \big[ \Prob(|C|=n) \big]^{1/n^\beta} = 
		\lim_{n\to\infty} \big[ \Prob(n \leq |C| < \infty) \big]^{1/n^\beta}
	\end{equation}
	exists for some $0<\nu<1$.
\end{definition}

We note that, for instance, for subcritical percolation the cluster of the 
origin has an exponential tail with~$\mu$ strictly between $0$ and~$1$ 
\cite[Theorems 6.78, 8.61, 8.65]{Grim99}. The result~\eqref{limit} also holds 
for critical and supercritical percolation, but with $\mu=1$, which indicates 
that the distribution decays slower than exponentially. For supercritical 
percolation, it is in fact known that $\Prob(|C|=n)$ decays like a stretched 
exponential with $\beta=(d-1)/d$. This result appears in~\cite{ACC90,Cerf99} 
for $d=2$, in~\cite{Cerf00} for $d=3$, and in~\cite{Cerf06} for $d\geq4$.

It is believed that the behaviour of the cluster-size distribution described 
above is rather typical. That is, \eqref{limit} is expected to be true quite 
generally for random fields in the non-percolating regime, i.e., when all 
clusters are finite. This result follows for instance if one can show that 
there exists a constant $A>0$ such that
\begin{equation}
\label{supermulti}
	\frac{\Prob(|C|=n+m)}{n+m} \geq A \, \frac{\Prob(|C|=n)}n \, 
	\frac{\Prob(|C|=m)}m \quad \text{for all }n,m\geq1.
\end{equation}
See~\cite[p.~91]{KS78} for an example of this supermultiplicativity result in 
the case of percolation, which is reproduced in~\cite[Lemma~6.102]{Grim99}.  
Furthermore, it is believed that~\eqref{superlimit} holds quite generally in 
the percolating regime with an exponent $\beta = (d-1)/d$. 

An example of a binary random field which does not satisfy the Markov property 
but does satisfy the boundary-$s$ Markov property both for $s=0$ and for 
$s=1$, is the random-cluster model~\cite{Grim06}. Although this model is 
defined in terms of edge occupation statuses rather than site occupation 
statuses (as considered here), it is not difficult to adapt our methods for 
the random-cluster model. Note that for the random-cluster model, the 
exponential decay in~\eqref{limit} can be shown by adapting the proof of 
Lemma~6.12 in~\cite{Grim99}, see the corrected version of Theorem~5.47 
in~\cite{Grim06}. Theorem~5.47 in~\cite{Grim06} is stated in the more general 
setting of finite-energy FKG measures satisfying the empty-boundary (i.e., 
boundary-$0$) Markov property. An important open problem for the 
random-cluster model is whether $\mu<1$ in the subcritical regime (see e.g.\ 
\cite[Conjecture~5.54]{Grim06}).

\subsection{Pattern theorems}
\label{ssec:pattern}

Throughout this paper, we will use~$Q$ to denote the cube of diameter~$r$ at 
the origin,
\begin{equation}
	Q = \{ x\in\Z^d : 0\leq x_i< r \text{ for all}~i = 1,2,\ldots,d \},
\end{equation}
where the diameter $r>0$ is considered to be fixed once and for all. The 
\emph{extended cube}~$\overline{Q}$ is obtained by extending~$Q$ by~$1$ unit 
in all directions, that is,
\begin{equation}
	\overline{Q} = \{ x\in\Z^d : -1\leq x_i< r+1 \text{ for all}~i = 
	1,2,\ldots,d \}.
\end{equation}
The \emph{boundary} of the cube~$Q$ is defined as $\partial Q = \overline{Q} 
\setminus Q$. Likewise, the boundary of~$\overline{Q}$ is defined as 
$\partial\overline{Q} = \overline{ \overline{Q} } \setminus \overline{Q}$, 
where, as before, $\overline{ \overline{Q} }$ is obtained by 
extending~$\overline{Q}$ by~$1$ unit in all directions. While~$Q$ always 
denotes the cube at the origin, we will write~$Q_x$ to denote the cube at the 
site~$x$, that is, $Q_x = \{x+y : y \in Q\}$. The boundary of  this cube, the 
extended cube at~$x$, and its boundary are defined analogously as $\partial 
Q_x = \{x+y : y \in \partial Q\}$, $\overline{Q}_x = \{x+y : y \in 
\overline{Q}\}$ and $\partial \overline{Q}_x = \{x+y : y \in \partial 
\overline{Q}\}$.

\begin{definition}[Pattern]
	A \emph{pattern}~$P$ is a prescribed configuration of the states of the 
	sites in the cube~$Q$, that is, $P = (P(x) : x\in Q)$ is an element of 
	$S^Q$.
\end{definition}

\begin{figure}
	\centering
	\begin{picture}(191,125)
	 \put(0,0){\includegraphics{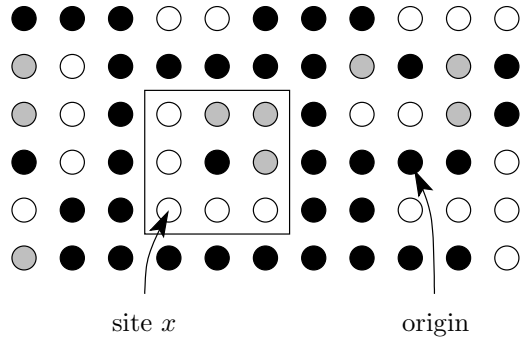}}
	 \put(50,2){\makebox(0,0)[b]{\smash{site $x$}}}
	 \put(159,2){\makebox(0,0)[b]{\smash{origin}}}
	\end{picture}
	\caption{\label{fig:PatternOnC}A piece of a three-state ($q=3$) 
	random-field configuration on~$\Z^2$, with sites in states $0$, $1$ 
	and~$2$ depicted in white, black and gray, respectively. The square 
	highlights a pattern of diameter $r=3$ occurring on the occupied cluster 
	of the origin at the site~$x$.}
\end{figure}

Suppose that~$P$ is a pattern. Then for a given configuration $\omega \in 
\Omega$, we say that $P$ \emph{occurs at the site~$x$} if $\omega(x+y) = P(y)$ 
for all $y \in Q$. We say that~$P$ occurs at~$x$ \emph{on~$C$} if~$P$ occurs 
at~$x$ and $\partial Q_x \subset C$. Thus, in our terminology, a pattern can 
only occur on the cluster of the origin if it is completely surrounded by this 
cluster. See Figure~\ref{fig:PatternOnC} for an illustration. The reason for 
defining occurrences of patterns on~$C$ in this way, is to guarantee that a 
pattern will always contribute the same number of sites to~$C$ whenever it 
occurs on~$C$. This fact is crucial for the two-pattern theorems stated below.  
However, the standard pattern theorem (Theorem~\ref{thm:pattern} below) still  
holds if we weaken our definition of an occurrence on~$C$, for instance if we  
say that a pattern~$P$ occurs at~$x$ on~$C$ if $P$ occurs at~$x$ and $\partial 
Q_x \cap C \neq \varnothing$. This is obvious, because there must be more 
occurrences of~$P$ on~$C$ under the weaker definition.

The standard pattern theorem states that for a given pattern~$P$, if the 
cluster of the origin has size~$n$, then, for some $a>0$ sufficiently small, 
it is very unlikely that~$P$ occurs less than~$an$ times on this cluster. This 
statement is true whether or not one allows occurrences of patterns to 
overlap. However, in this paper we will also study the ratio between the 
numbers of occurrences of two distinct patterns, and for these results it is 
important that patterns cannot overlap. We will avoid this by imposing a 
condition on which occurrences of a pattern are counted. Let us now explain 
this in more detail.

First of all, we note that there are patterns $P$ and~$P'$ for which overlaps 
are ruled out by definition. If that is the case, we do not have to impose 
restrictions on which occurrences of the patterns we count for the two-pattern 
theorems (Theorems \ref{thm:twopattern}  and~\ref{thm:strongtwopattern}) to 
hold. For generality, however, we will choose to consider only occurrences of 
patterns at sites of the grid $V := (r+2)\Z^d + (1,1,\ldots,1)$. That is, we 
denote by $N_P = N_P(\omega)$ the total number of occurrences of the 
pattern~$P$ on~$C$ at distinct sites of~$V$, and we will study~$N_P$ rather 
than the total number of distinct occurrences of~$P$ on~$C$. To explain our 
choice of~$V$, we note that later on, we are going to replace occurrences 
of~$P$ at sites of~$V$ by occurrences of~$P'$, and we want to guarantee that 
this does not change the state of the origin from occupied to non-occupied.  
This is why $V$ has been chosen such that the origin does not belong to~$Q_x$ 
for any $x \in V$.

Henceforth, we will write $\Prob_n$ for the measure $\Prob$ conditioned on the 
event $\{|C|=n\}$, that is, $\Prob_n( \,\cdot\, ) := \Prob\bigl( \,\cdot\, 
\bigm| |C|=n \bigr)$. In Section~\ref{sec:pattern} we will prove the following 
result, which is slightly different from the pattern theorem for lattice 
clusters appearing in~\cite{Madr99} because of the way we have defined~$N_P$:

\begin{theorem}[Pattern theorem]
\label{thm:pattern}
	Suppose that~$\Prob$ has the finite-energy property~\eqref{finiteE} and 
	that the cluster-size distribution satisfies~\eqref{limit}. Let~$P$ be a 
	pattern.  Then there exists an $a>0$ such that
	\[
		\limsup_{n\to\infty} \big[ \Prob_n(N_P \leq an) \big]^{1/n} < 1.
	\]
\end{theorem}

The main new results of this paper, however, concern ratios between the 
numbers of occurrences of two distinct patterns $P$ and~$P'$ on~$C$. In 
particular, we are interested in patterns~$P$ and~$P'$ such that one of these 
patterns contributes one more site to the cluster of the origin than the other 
pattern if it occurs on~$C$. We shall see that for such patterns, bounds on 
the ratio between the number of occurrences of~$P$ and the number of 
occurrences of~$P'$ lead directly to bounds on $\Prob(|C| = n+1) / \Prob(|C| = 
n)$ if the cluster-size distribution is known to have an exponential tail.

To state our results, for a given pattern~$P$, we shall write $c_P$ for the 
number of occupied sites in the pattern~$P$ which will be part of the cluster 
of the origin whenever~$P$ occurs on~$C$ at some site. We also introduce the 
notation
\begin{equation}
\label{PBox}
	\Prob\bigl( \Box P \bigr)
	:= \Prob\bigl( \omega(x) = P(x)\ \forall x\in Q, \omega(x) = 1\ \forall  
	x\in \partial Q \bigr)
\end{equation}
for the ``probability of an occurrence of~$P$ at the origin surrounded by an 
occupied cluster''.

Now suppose that the cluster-size distribution satisfies~\eqref{limit}. Then 
we define, for a given pattern $P$,
\begin{equation}
	\gamma_P := \mu^{-c_P} \, \Prob\bigl( \Box P \bigr),
\end{equation}
and, for distinct patterns $P$ and~$P'$,
\begin{equation}
	\label{gamma}
	\gamma_{PP'} := \frac{\gamma_P}{\gamma_{P'}} =
		\frac{\mu^{c_{P'}} \, \Prob\bigl( \Box P \bigr)}
		{\mu^{c_P} \, \Prob\bigl( \Box P' \bigr)}.
\end{equation}
We believe that a law of large numbers holds for patterns, stating that there 
exists a $\rho>0$ such that for any pattern~$P$, the number $N_P$ is 
concentrated around $\rho \, \gamma_P$ for ``almost all'' configurations.  
Although we cannot prove this, the following theorem (which we shall prove in 
Section~\ref{sec:twopattern}) implies that if one can prove a law of large 
numbers for one particular pattern, then it must hold for all patterns. It 
states that for ``almost all'' configurations, the ratio $N_P / N_{P'}$ is 
concentrated around $\gamma_{PP'} = \gamma_P / \gamma_{P'}$ for two distinct 
patterns $P$ and~$P'$:

\begin{theorem}[Two-pattern theorem]
\label{thm:twopattern}
	Let $P$ and $P'$ be two distinct patterns. Suppose that~$\Prob$ satisfies 
	the boundary-$1$ Markov property, and that the cluster-size distribution 
	satisfies~\eqref{limit}. Then for all $\epsilon>0$ and $\gamma_{PP'}$ as 
	defined in~\eqref{gamma},
	\[
		\limsup_{n \to \infty} \big[ \Prob_n( |N_P - \gamma_{PP'} \, N_{P'}| 
		\geq \epsilon n ) \big]^{1/n} < 1.
	\]
\end{theorem}

A natural question is whether stronger bounds on the ratio $N_P/N_{P'}$ hold, 
showing for instance that the difference between this ratio and $\gamma_{PP'}$ 
is at most of order $|C|^\alpha$ for some $0<\alpha<1$. Indeed, there are two 
cases in which we have obtained such stronger bounds. The first case is the 
case $c_P = c_{P'}$, that is, the case where the cluster size does not change 
if one replaces an occurrence of $P$ on~$C$ by an occurrence of $P'$ on~$C$.  
The second case is the case of stretched exponential decay of the cluster-size 
distribution (i.e., the supercritical case for percolation). In that case, our 
stronger bounds on $N_P - \gamma_{PP'} \, N_{P'}$ also lead to a stronger 
version of the ratio limit theorem, see Corollary~\ref{cor:strongratio} below.  
The stronger version of the two-pattern theorem reads as follows:

\begin{theorem}[Strengthened two-pattern theorem]
\label{thm:strongtwopattern}
	Consider a random field which has the boundary-$1$ Markov property, the 
	finite-energy property~\eqref{finiteE} and a cluster-size distribution 
	satisfying~\eqref{limit}. Suppose that $P$ and $P'$ are two distinct 
	patterns, and let $\gamma_{PP'}$ be defined as in~\eqref{gamma}.  Then the 
	following statements hold:
	\begin{enumerate}
	\item[(i)] If $c_P = c_{P'}$, then for all $\alpha > \ahalf$ and for every 
	$\epsilon > 0$,
	\[
		\limsup_{n \to \infty} \big[ \Prob_n( |N_P - \gamma N_{P'}| \geq 
		\epsilon n^\alpha) \big]^{1/n^{2\alpha-1}} < 1.
	\]
	\item[(ii)] If the cluster-size distribution has a stretched exponential 
	tail with exponent~$\beta$, then for every $\alpha \geq \ahalf(1+\beta)$ 
	and every $\epsilon>0$,
	\[
		\limsup_{n \to \infty} \big[ \Prob_n( N_P \geq \gamma N_{P'} + 
		\epsilon n^\alpha ) \big]^{1/n^{2\alpha-1}} < 1.
	\]
	\end{enumerate}
\end{theorem}

\paragraph{Remark.}
In Section~\ref{sec:twopattern} we show that 
Theorem~\ref{thm:strongtwopattern}(ii) can be strengthened if the 
supermultiplicativity result~\eqref{supermulti} holds.  
Under~\eqref{supermulti}, if $c_{P'} < c_P$ and the cluster-size distribution 
has a stretched exponential tail with exponent~$\beta$, one can show that for 
$\alpha = (2-\beta)^{-1}$ there exists an $a_0>0$ such that for every $a\geq 
a_0$,
\begin{equation}
	\limsup_{n \to \infty} \big[ \Prob_n(N_P \geq \gamma N_{P'} + a n^\alpha) 
	\big]^{1/n^{2\alpha-1}} < 1.
\end{equation}

Examples where Theorems \ref{thm:pattern}, \ref{thm:twopattern} 
and~\ref{thm:strongtwopattern}(i) apply are percolation, the Ising and Potts 
models, and, in general, the random-cluster measure. For 
Theorem~\ref{thm:strongtwopattern}(ii), stretched exponential decay is needed.

\subsection{Ratio limit theorems}
\label{ssec:ratio}

We can use our pattern theorems to prove that the ratio of the probabilities 
that the occupied cluster of the origin has sizes $n+1$ and~$n$ converges 
with~$n$. To prove this, one can in principle follow Kesten's argument 
in~\cite{Kest63}, where a ratio limit theorem is derived from a pattern 
theorem for self-avoiding walks. The same argument also appears in~\cite{MS93, 
Madr99}. It requires that the probability of a configuration changes by a 
constant factor whenever one replaces a single occurrence of a pattern~$P$ on 
the cluster of the origin by an occurrence of another pattern~$P'$. In our 
present context, this means that we need to assume the boundary-$1$ Markov 
property. However, under this condition we can use our two-pattern theorem 
(Theorem~\ref{thm:twopattern}) to give a more elegant and direct proof of the 
ratio limit theorem, avoiding the rather technical argument by Kesten. Our 
proof is presented in Section~\ref{sec:ratio}.

\begin{corollary}[Ratio limit theorem]
\label{cor:ratio}
	Suppose that~$\Prob$ has the finite-energy property~\eqref{finiteE} and 
	the boundary-$1$ Markov property, and that~\eqref{limit} holds for the 
	cluster-size distribution. Then the limit~$\mu$ in~\eqref{limit} also 
	satisfies
	\[
		\lim_{n\to\infty} \frac{\Prob(|C| = n+1)}{\Prob(|C| = n)} = 
		\lim_{n\to\infty} \frac{\Prob(n+1 \leq |C| < \infty)}{\Prob(n \leq |C| 
		< \infty)} = \mu.
	\]
\end{corollary}

Examples where Corollary~\ref{cor:ratio} applies are percolation, Ising and 
Potts models, as well as the random-cluster model for general $p$ and $q$. 
Note that in the case of percolation, for subcritical~$p$, the result in 
Corollary~\ref{cor:ratio} is stronger than~\eqref{limit}, but still much 
weaker than the widely believed tail-behaviour of the cluster-size 
distribution, namely that there exist $\theta = \theta(d) \in \R$ and $A = 
A(p,d)$ such that
\begin{equation}
	\Prob_p(|C| \geq n) = A n^\theta \mu^n [1+o(1)].
\end{equation}
For supercritical~$p$, we can obtain a stronger result than 
Corollary~\ref{cor:ratio} by virtue of our stronger version of the two-pattern 
theorem, Theorem~\ref{thm:strongtwopattern}:

\begin{corollary}[Strengthened ratio limit theorem]
\label{cor:strongratio}
	Suppose that~$\Prob$ has the boundary-$1$ Markov property and the 
	finite-energy property~\eqref{finiteE}, and that the cluster-size 
	distribution has a stretched exponential tail with exponent $\beta \in 
	(0,1)$. Then, for every $\epsilon>0$,
	\[
		\left| \frac{ \Prob(|C| = n+1) }{ \Prob(|C| = n)} - 1 \right|
		\leq \frac{\epsilon}{n^{(1-\beta)/2}}
	\]
	for sufficiently large~$n$. Hence, for every $x>0$ and $0 < \alpha \leq 
	(1-\beta)/2$,
	\[
		\lim_{n\to\infty} \frac{\Prob(|C| = n + \lfloor xn^\alpha 
		\rfloor)}{\Prob(|C| = n)} = 1.
	\]
\end{corollary}

\paragraph{Remark.}
If the supermultiplicativity result~\eqref{supermulti} holds, then we can 
derive a stronger lower bound for the ratio $\Prob(|C|=n+1) / \Prob(|C|=n)$, 
by virtue of the Remark below Theorem~\ref{thm:strongtwopattern}. Namely, in 
that case we can show that there exists a constant $A>0$ such that for~$n$ 
sufficiently large,
\begin{equation}
	\frac{ \Prob(|C| = n+1) }{ \Prob(|C| = n)} \geq 1 - 
	\frac{A}{n^{(1-\beta)/(2-\beta)}}.
\end{equation}

Examples where Corollary~\ref{cor:strongratio} applies are supercritical 
percolation, the Ising model and random-cluster measures in those cases where 
the Wulff-shape results have been proved (see \cite{ACC90, Cerf99, Cerf00, 
Cerf06} and the references therein).

\subsection{Consequences for maximal clusters}
\label{ssec:Gumbel}

In this subsection, we describe the consequences of the ratio limit theorem 
for maximal clusters as derived in~\cite{HR06}. In order to state our results, 
we need some further notation. Let $B_n = [-n,n]^d \cap \Z^d$ be the cube of 
width $2n+1$. We let
\begin{equation}
\label{meenzb}
	|\Cmax| = |\Cmax(\omega)| = \max_{x\in B_n}|C(x)|
\end{equation}
denote the size of the maximal cluster having a non-empty intersection 
with~$B_n$. Furthermore, we define the cluster $\Cle(x)$ by
\begin{equation}
\label{Cle}
	\Cle(x) = \begin{cases}
		C(x) & \text{if $x$ is the left-endpoint of~$C(x)$}, \\
		\varnothing & \text{otherwise},
	\end{cases}
\end{equation}
where by the left-endpoint of a finite set $A\subset\Z^d$, we mean the minimum 
of~$A$ in the lexicographic order. In~\cite{HR06}, results were shown for 
$|\Cmax|$ assuming the ratio limit theorem for the cluster $\Cle(0)$ instead 
of $C(0)$. Therefore, we shall need the following corollary, which we shall 
prove simultaneously with Corollary~\ref{cor:ratio} in 
Section~\ref{sec:ratio}:

\begin{corollary}[Ratio limit theorem for~$\Cle(0)$]
\label{cor:ratioCle}
	Suppose that~$\Prob$ has the finite-energy property~\eqref{finiteE} and 
	the boundary-$1$ Markov property, and that~\eqref{limit} holds. Then the 
	limit~$\mu$ in~\eqref{limit} also satisfies
	\[
		\lim_{n\to\infty} \frac{\Prob(|\Cle(0)| = n+1)}{\Prob(|\Cle(0)| = n)} 
		= \lim_{n\to\infty} \frac{\Prob(n+1 \leq |\Cle(0)| < \infty)}{\Prob(n 
		\leq |\Cle(0)| < \infty)} = \mu.
	\]
\end{corollary}

The results for~$|\Cmax|$ in~\cite{HR06} hold under further conditions on the 
measure~$\Prob$, that we will formulate now. We start by introducing a 
so-called `high mixing' condition. For $A\subseteq \Z^d$, we write $E_{A}$ for 
an event depending only on the site variables in~$A$. Let $\mathcal{F}_{A}$ 
denote the $\sigma$-field generated by the site variables in~$A$. For $m>0$, 
we define
\begin{equation}
	\label{mixfun}
	\phi(m) = \sup \frac1{|A_1|} \, \bigl| \Prob\left(E_{A_1} \mid E_{A_2} 
	\right) - \Prob\left(E_{A_1}\right) \bigr|,
\end{equation}
where the supremum is taken over all finite subsets $A_1,A_2$ of~$\Z^d$, with 
$d(A_1,A_2)\geq m$ ($d$ denoting Euclidian distance) and over all $E_{A_i} \in 
\mathcal{F}_{A_i}$ with $\Prob(E_{A_2}) > 0$.

Note that this $\phi(m)$ differs from the usual $\varphi$-mixing function 
since we divide by the size of the dependence set of the event~$E_{A_1}$. This 
is natural in the context of Gibbsian random fields, where the classical 
$\varphi$-mixing mostly fails (except for the simplest i.i.d.\ case and ad-hoc 
examples of independent copies of one-dimensional Gibbs measures). We are now 
ready to formulate the non-uniformly exponentially $\phi$-mixing (NUEM) 
condition:

\begin{definition}[NUEM]
	We say that a random field is \emph{non-uniformly exponentially 
	$\phi$-mixing} (NUEM) if there exist constants $C,c>0$ such that
	\begin{equation}
	\label{exponmix}
		\phi(m) \leq C \exp(-c m) \quad \text{for all } m>0.
	\end{equation}
\end{definition}

Examples of random fields satisfying the NUEM condition are Gibbs measures 
with exponentially decaying potential in the Dobrushin uniqueness regime, or 
local transformations of such measures. Of course, for site percolation, where 
we have independence, we have $\phi \equiv 0$. The NUEM condition is weaker 
than the more often used weak mixing (see \cite{Alex98,Alex04} for a 
definition of weak mixing and the related stronger notion of ratio weak 
mixing). Weak mixing holds for (i) the Ising model on $\Z^d$ for supercritical 
temperatures (\cite[Corollary~3.8]{Alex98}); (ii) the Ising model on $\Z^2$ 
for all temperatures and non-zero external field 
(\cite[Corollary~3.8]{Alex98}); (iii) the Potts model on $\Z^2$ with $q\geq 
26$ (\cite[Corollary~3.9]{Alex98} and \cite[Theorem~1.8]{Alex04}); for general 
Potts models on $\Z^d$ under the assumption of exponential decay of 
connectivities and random-cluster uniqueness; (iv) general random-cluster 
measures on $\Z^d$ under the assumption of exponential decay of connectivities 
and random-cluster uniqueness (\cite[Theorem~1.6]{Alex04}). We refer to 
\cite{Alex98, Alex04} and the references therein for a discussion on mixing 
aspects.

In the subcritical case, we shall deal with models where the cluster-size 
distribution has exponential tails, i.e., we shall assume that $\Prob$, the 
law of the random field, satisfies~\eqref{limit} with $\mu<1$. We shall also 
assume a \emph{second moment condition}, which is used in~\cite{HR06} to prove 
the asymptotics of the largest connected component. The precise assumption is 
that for all $\alpha>1$,
\begin{equation}
\label{secmcon}
	\limsup_{n\to\infty} \sum_{0<|x|<n^\alpha}
	\frac{\Prob\bigl( n\leq |\Cle(x)| < \infty,\,
			n\leq |\Cle(0)| < \infty \bigr)}
	{\Prob\bigl( n\leq |\Cle(0)| < \infty \bigr)} < \infty,
\end{equation}
In \cite[Proposition~3.7]{HR06}, it is shown that for Markov models with the 
FKG property, \eqref{secmcon} follows from~\eqref{limit} with $\mu<1$. An 
inspection of the proof of \cite[Proposition~3.7]{HR06} shows that it also 
applies to models with the FKG property satisfying the boundary-$s$ Markov 
property with $s=0$.

Let us now state our results for~$|\Cmax|$. As an implication of 
Corollary~\ref{cor:ratioCle}, we obtain the following result on the 
subcritical maximal cluster:

\begin{theorem}[Subcritical Gumbel maximal cluster]
\label{thm:Gumbel}
	Assume that $\Prob$ has the finite-energy property~\eqref{finiteE}, is 
	NUEM, and satisfies \eqref{limit} with $\mu<1$ and~\eqref{secmcon}. Then 
	there exists a sequence $u_n\in\N$, with $u_n\to\infty$, real numbers $a, 
	\rho>0$ and a bounded sequence $a_n\in[a,1]$, such that for all $x\in\N$,
	\[
		\Prob(|\Cmax| \leq u_n + x) = e^{-a_n \, \mu^x} + O(n^{-\rho}).
	\]
\end{theorem}

Theorem~\ref{thm:Gumbel} shows that $|\Cmax|$ is bounded above and below by 
Gumbel laws, and shows in particular that the sequence $|\Cmax| - u_n$ is 
tight. As explained in more detail in~\cite{HR06}, the statement above in 
terms of the sequence~$a_n$ is necessary, and is, for instance, also present 
when dealing with the maximum of $n$ i.i.d.\ geometric random variables.

Examples of random fields for which Theorem~\ref{thm:Gumbel} applies are given 
in the following corollary:

\begin{corollary}{\bf(Examples subcritical Gumbel maximal clusters)}
\label{cor-examplesdubGumbel}
	The conclusions in Theorem~\ref{thm:Gumbel} hold in the following special 
	cases:
	\begin{enumerate}
		\item[(i)] Subcritical percolation;
		\item[(ii)] Subcritical Ising model;
		\item[(iii)] Subcritical random-cluster models satisfying the 
		FKG-property, for which $\mu<1$.
	\end{enumerate}
\end{corollary}

The assumed FKG-property in Corollary~\ref{cor-examplesdubGumbel}(iii) is 
necessary to ensure that~\eqref{secmcon} applies. See e.g.~\cite{Grim06} for a 
discussion of the FKG-property for random-cluster measures, and \cite[Theorems 
5.55, 5.86]{Grim06} for examples of parameter values for which the 
random-cluster measure satisfies $\mu<1$.

Theorem~\ref{thm:Gumbel} follows from \cite[Theorem~3.6]{HR06}, and, in the 
case of percolation, from \cite[Theorem~1.1]{HR06} combined with 
\cite[Theorem~1.5]{HR06}. We note that in \cite{HR06}, it was assumed that the 
measure $\Prob$ has so-called \emph{subcritical clusters}\footnote{We take 
this opportunity to correct a mistake in Definition~3.3(i) in~\cite{HR06}: the 
condition that $\Prob( |\Cle(0)| < \infty )=1$ should read $\Prob( |C(0)| < 
\infty )=1$. Similarly in Definition~3.3(ii).} \cite[Definition~3.3(i)]{HR06}, 
which is implied by the ratio limit theorem with $\mu<1$, 
Corollary~\ref{cor:ratioCle} above.  In fact, Corollary~\ref{cor:ratioCle} 
implies that $\xi$ and~$\zeta$, defined in \cite[Definition~3.3(i)]{HR06}, are 
equal so that the strongest version in \cite[Theorem~3.6]{HR06} applies.

We now turn to the maximal supercritical cluster, which was investigated 
in~\cite{HR06} only in the context of site percolation, to which we will 
therefore restrict ourselves here as well. For supercritical~$p$, we define
\begin{equation}
\label{meenzbsup}
	|\Cmax| = \max_{x\in B_n: |C(x)| < \infty} |C(x)|
\end{equation}
to be the largest \emph{finite} cluster intersecting the cube. The ratio limit 
theorem implies that, in the language of \cite[Definition~3.3(ii)]{HR06}, 
percolation has supercritical clusters. Therefore, \cite[Theorem~3.9]{HR06} 
implies the following Gumbel statistics for the largest finite supercritical 
cluster:

\begin{theorem}{\bf(Supercritical Gumbel maximal cluster for percolation)}
\label{thm:Gumbelsup}
	Let $p_c<p<1$ and let $\Prob_p$ denote the percolation measure with 
	parameter~$p$. There exists a sequence~$u_n(x)$ with $u_n(x)\in\N$ and 
	$u_n(x)\to\infty$ for all $x\in\R$ as $n\to\infty$, such that for all 
	$x\in\R$,
	\[
		\Prob_p(|\Cmax| \leq u_n(x)) = e^{-e^{-x}} + o(1),
	\]
	where the error term may depend on~$x$.
\end{theorem}

\subsection{Further notation used in the proofs}
\label{ssec:notation}

Throughout the rest of the paper, we will write $c_n$ for the probability that 
the cluster of the origin has size~$n$, and $p_n$ for the probability that the 
cluster of the origin is finite and has size at least~$n$. That is, $c_n := 
\Prob( |C|=n )$ and $p_n := \Prob( n\leq|C|<\infty ) = \sum_{m=n}^\infty c_m$.  

\section{Proof of the pattern theorem}
\label{sec:pattern}

This section is devoted to the proof of the pattern theorem, 
Theorem~\ref{thm:pattern}. Our proof is similar to the proof in~\cite{Madr99} 
of a pattern theorem stated in a different context. We recall that the pattern 
theorem states that for a given pattern~$P$, there exists an $a>0$ such that 
it is very unlikely that $P$ occurs at less than $an$ distinct sites of the 
grid $V = (r+2)\Z^d + (1,1,\ldots,1)$ on a cluster of size~$n$.

The proof proceeds roughly as follows. If we assume that one is likely to see 
at most $an$ occurrences of~$P$ on~$C$ at sites of the grid~$V$, then there 
are many sites left on this grid where we can create new occurrences of~$P$ 
on~$C$. By creating a single new occurrence, we may change the size of~$C$, 
but never by  more than $|\overline{Q}| = (r+2)^d$ sites. Consider all 
configurations we can generate by introducing $\delta n$ new occurrences of 
$P$ on~$C$. For every configuration we start from, there is an exponential 
number of ways to introduce $\delta n$ occurrences of $P$ on~$C$, and all 
generated configurations contribute to $c_m$ for some $m \in [n - |Q|\delta n, 
n + |Q|\delta n]$, and hence to $p_{n - |Q| \delta n}$ in particular. However, 
introducing occurrences of~$P$ will change the probability, and it is 
furthermore clear that many of the generated configurations will be obtained 
multiple times from different starting configurations. Still, if the random 
field has the finite-energy property and $\delta$ is small enough, we will 
generate so many distinct configurations that this contribution wins, and this 
will prove Theorem~\ref{thm:pattern}. Now let us fill in the details of the 
proof.

\begin{proof}[Proof of Theorem~\ref{thm:pattern}]
	We want to study the probability that the cluster of the origin has 
	size~$n$, and $P$ occurs on~$C$ at no more than~$an$ distinct sites 
	of~$V$. So, for all $n \geq 1$ define the set of relevant configurations 
	by
	\begin{equation}
		S_n = \{ \omega \in \Omega: |C|=n, N_P \leq an \}.
	\end{equation}
	Consider a configuration $\omega \in S_n$. Then~$C$ has size~$n$, which 
	implies that there are at least $n/|\overline{Q}| =: bn$ extended cubes at 
	sites of~$V$ that intersect~$C$. No more than~$an$ of these cubes can 
	contain an occurrence of~$P$. Thus, assuming $a<b$, with every 
	configuration $\omega \in S_n$ we can associate a collection $X(\omega) 
	\subset V$ of exactly $bn-an$ sites such that the extended cubes at these 
	sites intersect~$C$, and $P$ does not occur at any of these sites.

	For any configuration $\omega \in S_n$ and site $x \in X(\omega)$, we can  
	create an occurrence of~$P$ at~$x$ on~$C$ by vacating all occupied sites 
	of $\partial \overline{Q}_x$ that do not belong to~$C$, occupying all 
	sites of~$\partial Q_x$, and changing the configuration inside~$Q_x$ into 
	an occurrence of~$P$ at~$x$. Here we emphasise that we first vacate sites 
	on the boundary of~$\overline{Q}_x$ to avoid connecting the occupied 
	cluster of the origin to another occupied cluster, whose size we cannot 
	control and might be infinite. Also note that by the finite-energy 
	property~\eqref{finiteE}, the factor by which the probability changes upon 
	a single introduction of~$P$ is bounded from below by a constant of the 
	form $\exp(-A \, |\partial \overline{Q} \cup \overline{Q}|) = \exp\bigl( 
	-A \, (r+4)^d \bigr)$ for some fixed $A>0$.

	Now, for every $\omega \in S_n$, consider  all possible ways of creating 
	occurrences of~$P$ at exactly $\delta n$ sites chosen from the 
	collection~$X( \omega )$, where $\delta > 0$ is a small number that will 
	be fixed later on. Write~$S'_n$ for the collection of all configurations 
	generated in this way. Then in general, the same configuration $\omega' 
	\in S'_n$ can be obtained from multiple configurations in~$S_n$. However, 
	the only differences between these configurations can be the local 
	configurations inside the extended cubes and their boundaries at those 
	sites of~$V$ where~$P$ occurs on the cluster~$C(\omega')$. One can have 
	$q^{| \overline{Q} \cup \partial \overline{Q} |}$ different configurations 
	inside an extended cube plus its boundary, where we recall that~$q$ is the  
	size of the state space~$S$ per site. Also, there are at most 
	$(\delta+a)n$ occurrences of~$P$ on~$C$ if we create $\delta n$ new 
	occurrences.  Therefore, on the one hand,
	\begin{equation}
	\label{pattern1}
		\Prob(S'_n) \geq \binom{bn-an}{\delta n}
		e^{-A \, |\overline{Q} \cup \partial \overline{Q}| \, \delta n} \, 
		q^{-(\delta+a)n \, |\overline{Q} \cup \partial \overline{Q}|} \, 
		\Prob(S_n).
	\end{equation}
	On the other hand it is clear that for any $\omega' \in S'_n$, the 
	occupied cluster of the origin is finite and has size at least $n - |Q| \, 
	\delta n$. Therefore,
	\begin{equation}
	\label{pattern2}
		\Prob(S'_n) \leq p_{n - |Q| \, \delta n}.
	\end{equation}

	We now divide \eqref{pattern1} and~\eqref{pattern2} by~$c_n$, combine the 
	two resulting inequalities, take $n$-th roots on both sides and then the 
	$\limsup_{n \to \infty}$, using~\eqref{limit} and Stirling's formula. This 
	leads to
	\begin{multline}
		\mu^{-|Q| \, \delta} \geq
		\frac{(b-a)^{b-a}}{\delta^\delta \, (b-a-\delta)^{b-a-\delta}} \,
		e^{-A \, |\overline{Q} \cup \partial \overline{Q}| \, \delta} \,
		q^{-(\delta+a) \, |\overline{Q} \cup \partial \overline{Q}|} \\
		\null \times \limsup_{n\to\infty}
		\left[ \Prob_n(N_P \leq an) \right]^{1/n}.
	\end{multline}
	At this point we may just as well take $a=\delta$. Setting $t=\delta / 
	(b-a)$ and $\tilde{\mu} = \mu^{-|Q|} \, e^{A |\overline{Q} \cup \partial 
	\overline{Q}|} \, q^{2|\overline{Q} \cup \partial \overline{Q}|}$, the 
	previous inequality can be rewritten as
	\begin{equation}
		\limsup_{n\to\infty} \left[ \Prob_n(N_P \leq \delta n) \right]^{1/n} 
		\leq \left( t^t \, (1-t)^{1-t} \, \tilde{\mu}^t \right)^{b-\delta}.
	\end{equation}
	The right-hand side is smaller than~$1$ whenever $t = \delta / (b-a) < 
	\tilde{\mu}^{-1}$.  Therefore, the left-hand side is smaller than~$1$ for 
	sufficiently small $\delta>0$, which proves Theorem~\ref{thm:pattern}.
\end{proof}

\section{Proofs of the two-pattern theorems}
\label{sec:twopattern}

In this section we are interested in the ratio between the numbers of 
occurrences of two distinct pattern $P$ and $P'$ on~$C$. We will prove, as 
stated in Theorems \ref{thm:twopattern} and~\ref{thm:strongtwopattern}, that 
if the random field has the boundary-$1$ Markov property, then the ratio 
$N_P/N_{P'}$ must be close to a fixed number $\gamma = \gamma_{PP'}$ defined 
in~\eqref{gamma}. The basic strategy of the proofs is as follows. We will 
consider those configurations such that $N_P-\gamma N_{P'}$ differs from~$0$ 
by at least $\epsilon n^\alpha$ occurrences for some $\epsilon>0$ and 
$0<\alpha\leq1$. Then we can make this difference smaller by either turning 
occurrences of~$P$ into occurrences of~$P'$ or the other way around. By 
deriving a bound on the probability of the collection of configurations 
generated in this way, in a similar way as in the proof of 
Theorem~\ref{thm:pattern}, we can then show that the probability that $|N_P - 
\gamma N_{P'}|$ is at least $\epsilon n^\alpha$ must be small. We start by 
proving Theorem~\ref{thm:twopattern}, in which $\alpha=1$, after which we 
prove Theorem~\ref{thm:strongtwopattern}.

\begin{proof}[Proof of Theorem~\ref{thm:twopattern}]
	We will show that under the conditions stated in the theorem, for every 
	choice of $P$ and~$P'$ (and every possible value of the corresponding 
	$\gamma = \gamma_{PP'}$), and for all $\epsilon>0$,
	\begin{equation}
	\label{twopattern1}
		\limsup_{n\to\infty}
		\big[ \Prob_n( N_P \leq \gamma N_{P'} - \epsilon n) \big]^{1/n}
		\leq \frac{ \bigl( 1+\frac{\epsilon}{1+\gamma}
					\bigr)^{1+(1+\epsilon) \, \gamma^{-1}}}
				{(1+\epsilon)^{(1+\epsilon) \, \gamma^{-1}}}
		< 1.
	\end{equation}
	Since this holds for any choice of the two patterns, we can interchange 
	the roles of $P$ and~$P'$ in~\eqref{twopattern1}, which also replaces 
	$\gamma = \gamma_{PP'}$ by $\gamma^{-1} = \gamma_{P'P}$ and $\epsilon$ by 
	$\epsilon \gamma^{-1}$, to obtain
	\begin{equation}
	\label{twopattern2}
		\limsup_{n\to\infty}
		\big[ \Prob_n( N_P \geq \gamma N_{P'} + \epsilon n) \big]^{1/n}
		\leq \frac{ \bigl( 1+\frac{\epsilon}{1+\gamma}
					\bigr)^{1+\gamma+\epsilon}}
				{\bigl( 1+\frac{\epsilon}{\gamma} \bigr)^{\gamma+\epsilon}}
		< 1.
	\end{equation}
	These two results together evidently imply Theorem~\ref{thm:twopattern}.  
	Hence, it suffices to derive~\eqref{twopattern1}.

	To derive~\eqref{twopattern1}, let $S_n$ be the collection of 
	configurations such that $|C|=n$ and $N_P \leq \gamma N_{P'} - \epsilon 
	n$. For every $\omega \in S_n$, consider all possible ways in which we can 
	change $\delta \, \epsilon n$ occurrences of $P'$ on~$C$ at sites of~$V$ 
	into occurrences of $P$ on~$C$, where $0 < \delta < 1$ will be fixed later 
	on. Then, for given $N_P$ and $N_{P'}$, the ratio of the number of 
	configurations generated in this way to the number of configurations from 
	which they were obtained, is given by the factor
	\begin{equation}
	\begin{split}
	\label{twopatternfactor}
		&\binom{N_P + N_{P'}}{N_P + \delta \, \epsilon n} \,
		\binom{N_P + N_{P'}}{N_P}^{-1} \\
		&\qquad\qquad = \frac{N_P!}{(N_P + \delta \, \epsilon n)!} \, 
		\frac{N_{P'}!}{(N_{P'} - \delta \, \epsilon n)!} \\
		&\qquad\qquad \geq \frac{N_P!}{(N_P + \delta \, \epsilon n)!} \, 
		\frac{(\gamma^{-1} N_P + \gamma^{-1} \, \epsilon n)!}{(\gamma^{-1} N_P 
		+ (\gamma^{-1}-\delta) \, \epsilon n)!},
	\end{split}
	\end{equation}
	where the last inequality follows from $N_P \leq \gamma N_{P'} - \epsilon 
	n$.

	Now we observe (by differentiating with respect to~$N$) that for all 
	$k=1,2,\ldots,\delta \, \epsilon n$, the fraction $(\gamma^{-1}N + 
	(\gamma^{-1} - \delta) \, \epsilon n + k) / (N + k)$ is non-increasing 
	in~$N$. Therefore, using the fact that $N_P$ is necessarily smaller 
	than~$n$, we can bound the factor~\eqref{twopatternfactor} by
	\begin{equation}
	\begin{split}
		&\prod_{k=1}^{\delta \, \epsilon n}
		\frac{\gamma^{-1}N_P + (\gamma^{-1} - \delta) \, \epsilon n + k}
			{N_P + k} \\
		&\qquad\qquad \geq \prod_{k=1}^{\delta \, \epsilon n}
		\frac{\gamma^{-1}n + (\gamma^{-1} - \delta) \, \epsilon n + k}
			{n + k} \\
		&\qquad\qquad = \frac{n!}{(n + \delta \, \epsilon n)!} \,
		\frac{(\gamma^{-1} n + \gamma^{-1} \, \epsilon n)!}
			{(\gamma^{-1} n + (\gamma^{-1}-\delta) \, \epsilon n)!}.
	\end{split}
	\end{equation}

	For every configuration obtained from $\omega\in S_n$ by changing $\delta   
	\, \epsilon n$ occurrences of $P'$ on~$C$ into occurrences of $P$ on~$C$, 
	the cluster of the origin has size $|C| = n - \delta \, \epsilon n \, 
	\delta_c$, where $\delta_c = c_{P'} - c_P$. Moreover, by virtue of the 
	boundary-$1$ Markov property, every change of $P'$ on~$C$ into~$P$ changes 
	the probability of a configuration by the factor $\Prob(\Box P) / 
	\Prob(\Box P')$, where we recall the definition~\eqref{PBox} of 
	$\Prob(\Box P)$. Therefore, we can write
	\begin{multline}
		\frac{c_{n - \delta \, \epsilon n \, \delta_c}}{c_n} \geq
		\left( \frac{\Prob(\Box P)}{\Prob(\Box P')}
		\right)^{\delta \, \epsilon n}
		\frac{n!}{(n + \delta \, \epsilon n)!} \,
		\frac{(\gamma^{-1} n + \gamma^{-1} \, \epsilon n)!}
			{(\gamma^{-1} n + (\gamma^{-1}-\delta) \, \epsilon n)!} \\
		\null \times \Prob_n(N_P \leq \gamma N_{P'} - \epsilon n).
	\end{multline}
	Taking the $n$-th root on both sides and then the $\limsup_{n\to\infty}$, 
	using~\eqref{limit}, \eqref{gamma} and Stirling's formula, leads to
	\begin{equation}
	\label{twopatternf1}
		1 \geq f_\delta(\epsilon) \, \limsup_{n\to\infty}
		\left[ \Prob_n(N_P \leq \gamma N_{P'} - \epsilon n)
		\right]^{1/n},
	\end{equation}
	where
	\begin{equation}
		f_\delta(x) =
		\frac{(1+x)^{(1+x)\gamma^{-1}}}
		{(1 + (1-\delta\gamma) x)^{\gamma^{-1} + (\gamma^{-1}-\delta) x}
			\, (1 + \delta x)^{1 + \delta x}}.
	\end{equation}

	Observe that $f_\delta(0)=1$, and taking the derivative of $f_\delta(x)$ 
	with respect to~$x$ yields
	\begin{equation}
		\frac{f_\delta'(x)}{f_\delta(x)} = \frac{1}{\gamma} \,
		\log\left( \frac{1+x}{1+x-\delta\gamma x} \right)
		+ \delta \, \log\left( \frac{1+x-\delta\gamma x}{1+\delta x} \right).
	\end{equation}
	It follows that for all $x>0$ and $\delta \leq (1+\gamma)^{-1}$, 
	$f_\delta'(x) > 0$ and therefore $f_\delta(x) > 1$. It is not difficult to 
	see that for fixed $x>0$, $f_\delta(x)$ is actually maximal at $\delta = 
	(1+\gamma)^{-1}$. We therefore set $\delta = (1+\gamma)^{-1}$, and then 
	the desired result~\eqref{twopattern1} follows from~\eqref{twopatternf1}.  
	This proves Theorem~\ref{thm:twopattern}.
\end{proof}

\paragraph{Remark.}
The bound appearing on the right-hand side of~\eqref{twopattern1} is not best 
possible in general. Indeed, if we restrict $\epsilon$ to the range 
$(0,\gamma)$, then the combinatorial factor~\eqref{twopatternfactor} can also 
be bounded by
\begin{equation}
\begin{split}
	\frac{N_P!}{(N_P + \delta \, \epsilon n)!} \,
	\frac{N_{P'}!}{(N_{P'} - \delta \, \epsilon n)!}
	&\geq \frac{(\gamma N_{P'} - \epsilon n)!}
			{(\gamma N_{P'} - (1-\delta) \, \epsilon n)!} \,
		\frac{N_{P'}!}{(N_{P'} - \delta \, \epsilon n)!} \\
	&\geq \frac{(\gamma n - \epsilon n)!}
			{(\gamma n - (1-\delta) \, \epsilon n)!} \,
		\frac{n!}{(n - \delta \, \epsilon n)!}.
\end{split}
\end{equation}
Here, as before, the first inequality follows from $N_P \leq \gamma N_{P'} - 
\epsilon n$ and the second is a consequence of the fact that $N_{P'}$ is 
necessarily less than~$n$. The same reasoning as in the previous proof then 
leads us again to an inequality of the form~\eqref{twopatternf1}, where now 
the function $f_\delta(x)$ is given by
\begin{equation}
	f_\delta(x) =
	\frac{(1 - x\gamma^{-1})^{\gamma-x}}
		{(1 - (1-\delta) x\gamma^{-1})^{\gamma - (1-\delta) x} \,
			(1 - \delta x)^{1 - \delta x}}.
\end{equation}
As before, for fixed $0<x<\gamma$, this function is maximal and larger 
than~$1$ at $\delta = (1+\gamma)^{-1}$. This leads for $0<\epsilon<\gamma$ to 
the bound
\begin{equation}
	\limsup_{n\to\infty}
	\big[ \Prob_n(N_P \leq \gamma N_{P'} - \epsilon n) \big]^{1/n}
	\leq \frac{ \bigl( 1-\frac{\epsilon}{1+\gamma}
				\bigr)^{1+\gamma-\epsilon}}
			{\bigl( 1-\frac{\epsilon}{\gamma} \bigr)^{\gamma-\epsilon}}
	< 1.
\end{equation}
This bound is better than the bound in~\eqref{twopattern1} for small values 
of~$\gamma$, but worse than~\eqref{twopattern1} for large values of~$\gamma$.

\begin{proof}[Proof of Theorem~\ref{thm:strongtwopattern}]
	The proof of Theorem~\ref{thm:strongtwopattern} is similar to the proof of 
	Theorem~\ref{thm:twopattern} above. We define the set of relevant 
	configurations by
	\begin{equation}
		S_n = \{ \omega \in \Omega : |C| = n, N_P - \gamma N_{P'}  \geq 
		\epsilon n^\alpha \}.
	\end{equation}
	For every $\omega \in S_n$, we can choose $\delta \, \epsilon n^\alpha$ 
	occurrences of $P$ on~$C$ at sites of~$V$ and turn them into occurrences 
	of~$P'$, where $0 < \delta < 1$ is a number that we will fix later. This 
	will change the size of~$C$ to $n + \delta \, \epsilon n^\alpha \, 
	\delta_c$, since $\delta_c = c_{P'} - c_P$ is the change in~$|C|$ if we 
	replace a single occurrence of $P$ on~$C$ by an occurrence of~$P'$.

	For given $N_P$ and $N_{P'}$, the ratio of the number of generated 
	configurations in which there are $N_P - \delta \, \epsilon n^\alpha$ 
	occurrences of $P$ on~$C$ at sites of~$V$ and $N_{P'} + \delta \, \epsilon 
	n^\alpha$ occurrences of $P'$ on~$C$ at sites of~$V$ to the number of 
	configurations in~$S_n$ from which they can be generated, is, similarly 
	to~\eqref{twopatternfactor}, given by
	\begin{equation}
	\begin{split}
		&\binom{ N_P+N_{P'} }{ N_P - \delta \, \epsilon n^\alpha }
		\binom{ N_P+N_{P'} }{ N_P }^{-1} \\
		&\qquad\qquad = \frac{N_P!}{(N_P - \delta \, \epsilon n^\alpha)!} \,
			\frac{N_{P'}!}{(N_{P'} + \delta \, \epsilon n^\alpha)} \\
		&\qquad\qquad \geq \frac{(\gamma N_{P'} + \epsilon n^\alpha)!}
				{(\gamma N_{P'} + (1-\delta) \epsilon n^\alpha)!} \,
			\frac{N_{P'}!}{(N_{P'} + \delta \, \epsilon n^\alpha)!} \\
		&\qquad\qquad \geq \frac{(\gamma n + \epsilon n^\alpha)!}
				{(n - \delta \, \epsilon n^\alpha)!} \,
			\frac{n!}{(n + \delta \, \epsilon n^\alpha)!}.
	\end{split}
	\end{equation}
	Here, the first inequality follows from $N_P - \gamma N_{P'} \geq \epsilon 
	n^\alpha$, and in the second inequality we used monotonicity of the 
	expression in~$N_{P'}$, which can be shown by an argument similar to the 
	one below~\eqref{twopatternfactor}, together with the fact that $N_{P'} 
	\leq n$. We note that we have assumed that $0 < \delta \leq \gamma^{-1}$ 
	to obtain the second inequality.

	Because for every configuration generated from $\omega\in S_n$, we have 
	that $|C| = n+\delta \, \epsilon n^\alpha \, \delta_c$, we can write
	\begin{multline}
		\frac{c_{n + \delta \, \epsilon n^\alpha \, \delta_c}}{c_n}
		\geq
		\left( \frac{\Prob(\Box P')}{\Prob(\Box P)}
		\right)^{\delta\,\epsilon n^\alpha}
		\frac{(\gamma n + \epsilon n^\alpha)!}
			{(n - \delta \, \epsilon n^\alpha)!} \,
		\frac{n!}{(n + \delta \, \epsilon n^\alpha)!} \\
		\null \times \Prob_n(N_P-\gamma N_{P'} \geq \epsilon n^\alpha).
	\end{multline}
	Using Stirling's formula and substituting~\eqref{gamma} for $\gamma = 
	\gamma_{PP'}$, this expression can be rewritten as
	\begin{multline}
		\frac{c_{n + \delta \, \epsilon n^\alpha \, \delta_c}}{c_n}
		\geq \mu^{\delta \, \epsilon n^\alpha \, \delta_c} \,
		e^{(\gamma^{-1}\delta - \ahalf\gamma^{-1}\delta^2 - \ahalf\delta^2)
			\epsilon n^{2\alpha - 1} + o(n^{2\alpha - 1})} \\
		\null \times \Prob_n(N_P-\gamma N_{P'} \geq \epsilon n^\alpha).
	\end{multline}
	We now choose $\delta = (1 + \gamma)^{-1}$, since this maximizes the 
	middle stretched exponential term on the right-hand side. Raising to the 
	power $n^{1-2\alpha}$ (here we assume $\alpha > \ahalf$) and taking the 
	$\limsup_{n\to\infty}$ then leads to
	\begin{multline}
	\label{strongtwopattern}
		e^{-\ahalf(\gamma+\gamma^2)^{-1} \, \epsilon} \,
		\limsup_{n\to\infty}
		\left[
			\frac{c_{n + \delta \, \epsilon n^\alpha \, \delta_c}}{c_n} \,
			\mu^{-\delta \, \epsilon n^\alpha \, \delta_c}
		\right]^{1/n^{2\alpha-1}} \\
		\geq
		\limsup_{n\to\infty}
		\left[
			\Prob_n(N_P - \gamma N_{P'} \geq \epsilon n^\alpha)
		\right]^{1/n^{2\alpha-1}}.
	\end{multline}
	Observe that the left-hand side in~\eqref{strongtwopattern} is smaller 
	than~$1$ if $\delta_c = 0$,  which immediately proves the first claim of 
	the theorem.

	For $\delta_c \neq 0$ we need a bound on the ratio of cluster-size 
	probabilities appearing on the left-hand side of~\eqref{strongtwopattern}.  
	The exponential decay~\eqref{limit} does not lead to useful bounds, but in 
	the supercritical case, where $\mu=1$, the stretched exponential 
	decay~\eqref{superlimit} tells us that
	\begin{equation}
		\limsup_{n\to\infty} \left[ \frac{c_{n + \delta \, \epsilon n^\alpha 
		\, \delta_c}}{c_n} \right]^{1/n^\beta} = 1.
	\end{equation}
	In~\eqref{strongtwopattern}, this makes the left-hand side smaller 
	than~$1$ if we take $2\alpha - 1 \geq \beta$, that is, $\alpha \geq \ahalf 
	(1+\beta)$. This completes the proof of 
	Theorem~\ref{thm:strongtwopattern}.
\end{proof}

To conclude this section, we strengthen the above bound in the case $\delta_c 
< 0$ under the condition~\eqref{supermulti} of supermultiplicativity. This 
gives
\begin{equation}
	\frac{c_{n - \delta \, \epsilon n^\alpha \, |\delta_c|}}{c_n} \leq \frac1A 
	\, \frac{\delta \, \epsilon n^\alpha \, |\delta_c|}{ c_{\delta \, \epsilon 
	n^\alpha \, |\delta_c|} }.
\end{equation}
Using the stretched exponential decay~\eqref{superlimit}, we obtain from this 
inequality
\begin{equation}
	\limsup_{n\to\infty} \left[ \frac{ c_{n - \delta \, \epsilon n^\alpha \, 
	|\delta_c|} }{ c_n } \right]^{1/n^{\alpha\beta}} \leq e^{\xi (\delta \, 
	\epsilon \, |\delta_c|)^{\beta} }
\end{equation}
for some $0 < \xi < \infty$. We note that $\alpha\beta = 2\alpha-1$ if we take 
$\alpha = (2-\beta)^{-1}$. Then, if we insert the previous expression 
into~\eqref{strongtwopattern}, we see that the left-hand side becomes smaller 
than~$1$ if we take $\epsilon$ larger than some $a_0 > 0$, which proves the 
statement in the Remark following Theorem~\ref{thm:strongtwopattern}.

\section{Proofs of the ratio limit theorems}
\label{sec:ratio}

We shall now show how the pattern theorems proved above can be combined to 
derive the ratio limit theorems (Corollaries \ref{cor:ratio} 
and~\ref{cor:strongratio}). For this, we shall take a fixed cube diameter 
$r=3$, and consider two specific patterns $P$ and~$P'$: $P$ is the pattern 
such that the site $(1,1,\ldots,1)$ is occupied and all other sites of~$Q$ are 
vacant, and $P'$ is the pattern such that the origin is occupied and all other 
sites of~$Q$ are vacant, see Figure~\ref{fig:patterns}.  For these two 
patterns and integers $i,j\geq 0$, we introduce the notation
\begin{equation}
\label{c_n(i,j)}
	c_n(i,j) = \Prob_p(|C|=n, N_P = i, N_{P'} = j).
\end{equation}
Observe that in our earlier notation used to formulate the pattern theorems, 
$\delta_c = c_{P'} - c_P = 1$. For percolation, the patterns $P$ and $P'$ are 
chosen such that whenever we change an occurrence of $P$ on~$C$ into an 
occurrence of $P'$ on~$C$, we do not change the probability of the 
configuration. This is, however, not generally the case for a Markovian random 
field.

\begin{figure}
	\centering
	\begin{picture}(200,66)
		\put(0,0){\includegraphics{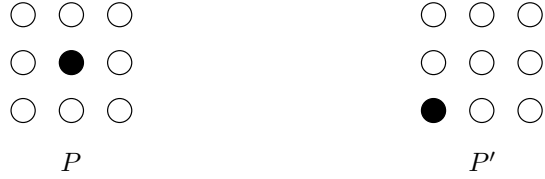}}
		\put(23,2){\makebox(0,0)[b]{\smash{$P$}}}
		\put(177,2){\makebox(0,0)[b]{\smash{$P'$}}}
	\end{picture}
	\caption{\label{fig:patterns}The patterns $P$ and~$P'$ in dimension $d=2$.  
	White sites are vacant, black sites are occupied.}
\end{figure}

Our proofs of the ratio limit theorems are based on the following observation.  
Let $S_n(i,j)$ be the collection of configurations such that $|C|=n$, $N_P = 
i$ and $N_{P'} = j$. Then, for every $\omega \in S_n(i,j)$, there are $i$~ways 
of changing one occurrence of $P$ on~$C$ at a site of~$V$ into an occurrence 
of $P'$ on~$C$, leading to an $\omega' \in S_{n+1}(i-1, j+1)$. For each 
configuration that we generate in this way, there are~$j$ other configurations 
in the set $S_n(i,j)$ from which we could have obtained the same configuration 
by changing one occurrence of $P$ on~$C$ at a site of~$V$ into an occurrence 
of $P'$ on~$C$.  Therefore, for $i\geq1$ and $j\geq0$,
\begin{equation}
\label{changePtoP'}
	c_{n+1}(i-1,j+1) = \frac{i}{j+1} \, \frac{\Prob(\Box P')}{\Prob(\Box P)} 
	\, c_n(i,j),
\end{equation}
where we have again used the boundary-$1$ Markov property. From this equality 
one sees that bounds on the ratio $N_P / N_{P'}$ will lead directly to bounds 
on $c_{n+1} / c_n$. We will now use this to prove Corollary~\ref{cor:ratio} 
for random fields satisfying the boundary-$1$ Markov property.

\begin{proof}[Proof of Corollaries \ref{cor:ratio} and~\ref{cor:ratioCle}]
	First let us show that if $c_{n+1} / c_n$ converges to~$\mu$ and $p_n = 
	\sum_{m=n}^\infty c_m$, then $p_{n+1} / p_n$ converges to~$\mu$ as well.  
	To show this, fix $0 < \epsilon < \mu$. Since $c_{n+1} / c_n$ converges 
	to~$\mu$, there exists an integer $N_\epsilon > 0$ such that
	\begin{equation}
		\left| \frac{c_{n+1}}{c_n}-\mu \right| \leq \epsilon
	\end{equation}
	for all $n > N_\epsilon$. Observing that
	\begin{equation}
		\frac{c_n}{p_n} = \left[ \sum_{m=0}^\infty \frac{c_{n+m}}{c_n} 
		\right]^{-1} = \left[ 1 + \sum_{m=1}^\infty \prod_{k=1}^m 
		\frac{c_{n+k}}{c_{n+k-1}} \right]^{-1},
	\end{equation}
	it follows that for $n > N_\epsilon$, if $\mu < 1$,
	\begin{align}
	\label{bound1}
		\frac{c_n}{p_n} &\leq \left[ \sum_{m=0}^\infty (\mu-\epsilon)^m 
		\right]^{-1} = 1-\mu+\epsilon, \\
		\frac{c_n}{p_n} &\geq \left[ \sum_{m=0}^\infty (\mu+\epsilon)^m 
		\right]^{-1} = 1-\mu-\epsilon.
	\end{align}
	Since $p_{n+1}/p_n = 1-c_n/p_n$, this implies $\lim_{n \to \infty} p_{n+1} 
	/ p_n = \mu$ for every $0 < \mu < 1$. Now note that \eqref{bound1} also 
	applies if $\mu=1$. Since $c_n/p_n \geq 0$ holds trivially, we also obtain 
	$\lim_{n \to \infty} p_{n+1} / p_n = \mu$ when $\mu=1$. Thus, to prove 
	Corollary~\ref{cor:ratio}, it remains to show that $c_{n+1} /c_n$ 
	converges to~$\mu$.

	Now let us write $c^\ast_n = \Prob\bigl( |\Cle(0)| = n \bigr)$ and 
	$p^\ast_n = \sum_{m=n}^\infty c^\ast_m$. We can repeat the argument above 
	(with $c^\ast_n, p^\ast_n$ in place of $c_n, p_n$) to see that 
	$c^\ast_{n+1} / c^\ast_n \to\mu$ implies $p^\ast_{n+1} / p^\ast_n \to\mu$.  
	However, for translation-invariant~$\Prob$, we have that $c_n = n \, 
	c^\ast_n$ for all $n\geq1$ (this is Lemma~4.1 in~\cite{HR06}), so that 
	$c^\ast_{n+1} / c^\ast_n \to\mu$ is implied by $c_{n+1} / c_n \to\mu$. We 
	conclude that establishing that $c_{n+1} / c_n$ converges to~$\mu$ 
	suffices to prove not only Corollary~\ref{cor:ratio} but also 
	Corollary~\ref{cor:ratioCle}.
		
	We will show that $c_{n+1} / c_n$ converges to~$\mu$ for a random field 
	satisfying the boundary-$1$ Markov property by using Theorems 
	\ref{thm:pattern} and~\ref{thm:twopattern}. Let the patterns $P$ and~$P'$ 
	be as above, and let $a>0$ be the constant appearing in 
	Theorem~\ref{thm:pattern} for the pattern~$P'$. Using the 
	notation~\eqref{c_n(i,j)} introduced above, we can write
	\begin{equation}
	\label{ratioeq}
		\frac{c_{n+1}}{c_n} = \sum_{j = -1}^n \sum_{i = 1}^n 
		\frac{c_{n+1}(i-1,j+1)}{c_n}.
	\end{equation}
	We may use Theorems \ref{thm:pattern} and~\ref{thm:twopattern} together 
	with~\eqref{limit}, to restrict the sums in~\eqref{ratioeq} at the cost of 
	introducing an exponentially small error term. For convenience, we write 
	$\gamma = \gamma_{PP'}$ for the constant defined in~\eqref{gamma}.  
	Applying our observation~\eqref{changePtoP'}, we obtain
	\begin{equation}
	\begin{split}
		\frac{c_{n+1}}{c_n}
		&= \sum_{j = an}^n \,
			\sum_{i = \gamma j - o(n)}^{\gamma j + o(n)}
			\frac{c_{n+1}(i-1,j+1)}{c_n} + o(1) \\
		&= \sum_{j = an}^n \,
			\sum_{i = \gamma j - o(n)}^{\gamma j + o(n)}
			\frac{i}{j+1} \, \frac{\Prob(\Box P')}{\Prob(\Box P)} \, 
			\frac{c_n(i,j)}{c_n} + o(1) \\
		&= \sum_{j = an}^n \,
			\sum_{i = \gamma j - o(n)}^{\gamma j + o(n)}
			\big(\gamma + o(1)\big) \, \frac{\Prob(\Box P')}{\Prob(\Box P)} \,
			\frac{c_n(i,j)}{c_n} + o(1) \\
		&= \big(\mu + o(1)\big) \sum_{j = 0}^n \, \sum_{i = 0}^n
			\frac{c_n(i,j)}{c_n} + o(1) \\
		&= \mu+o(1).
	\end{split}
	\end{equation}
	This proves that the ratio $c_{n+1} / c_n$ converges to~$\mu$, which 
	completes the proof of Corollaries \ref{cor:ratio} 
	and~\ref{cor:ratioCle}.
\end{proof}

\begin{proof}[Proof of Corollary~\ref{cor:strongratio}]
	Using the same notations as above, by Theorem~\ref{thm:strongtwopattern} 
	we can write
	\begin{equation}
	\label{strongratioproof}
		\frac{c_{n+1}}{c_n} = \sum_{j = an}^n \, {\sum_i}' \frac{i}{j+1} \,
		\frac{\Prob(\Box P')}{\Prob(\Box P)} \,
		\frac{c_n(i,j)}{c_n} + o\bigl( n^{-(1-\beta)/2} \bigr),
	\end{equation}
	where the prime on the second sum means that $i$ is restricted to run from 
	$\gamma j - \epsilon n^{(1+\beta)/2}$ to $\gamma j + \epsilon 
	n^{(1+\beta)/2}$, with arbitrary $\epsilon>0$. Proceeding as in the proof 
	of Corollary~\ref{cor:ratio}, this implies that
	\begin{equation}
		\left| \frac{c_{n+1}}{c_n} - 1 \right| \leq 
		\frac{2\epsilon}{an^{(1-\beta)/2}}
	\end{equation}
	for sufficiently large~$n$, as required. This is turn implies
	\begin{equation}
		\frac{c_{n+\lfloor xn^\alpha \rfloor}}{c_n}
		= \prod_{k=1}^{\lfloor xn^\alpha \rfloor} \frac{c_{n+k}}{c_{n+k-1}}
		= \left[ 1 + o\bigl( n^{-(1-\beta)/2} \bigr)
			\right]^{\lfloor xn^\alpha \rfloor}
		= 1+o(1)
	\end{equation}
	for all $x>0$ and $0 < \alpha \leq (1-\beta)/2$, establishing 
	Corollary~\ref{cor:strongratio}.
\end{proof}

We close off this section by proving the statement in the Remark below 
Corollary~\ref{cor:strongratio}. By the Remark following 
Theorem~\ref{thm:strongtwopattern}, under~\eqref{supermulti} we may 
restrict~$i$ in the primed sum in~\eqref{strongratioproof} to the range 
$[\gamma j - a_0 n^{1/(2-\beta)},n]$. The same reasoning as before then leads 
to the desired result
\begin{equation}
	\frac{c_{n+1}}{c_n} \geq 1 - \frac{2a_0}{an^{(1-\beta)/(2-\beta)}}.
\end{equation}

\paragraph{Acknowledgements.}
The work of RvdH was supported in part by the Netherlands Organization for 
Scientific Research (NWO). We thank Geoffrey Grimmett for clarifying the proof 
of~\eqref{limit} in the boundary-$0$ Markov case, as well as for pointing us 
to \cite{Alex98, Alex04} on mixing conditions.

\end{document}